\documentclass[12pt]{article}
\usepackage{amsmath,amssymb,amsthm,amsfonts}
\usepackage{color}
\usepackage{calligra}
\usepackage[T1]{fontenc}
\usepackage{graphicx}
\def\RR{\mathbb R}
\def\ZZ{\mathbb Z} 
\def\CC{\mathbb C} 
\def\NN{\mathbb N} 
\def\phi{\varphi}

\newtheorem{theorem}{Theorem}[section]
\newtheorem{lemma}[theorem]{Lemma}

\numberwithin{equation}{section}
\title{A Collocation Method in Spline Spaces  for the Solution of Linear Fractional Dynamical Systems}
\author{Enza Pellegrino\thanks{\noindent {\it Dip. DIIE, Universit\`{a} de L'Aquila }, Via G. Gronchi 18, 67100  L'Aquila, Italy. e-mail: \tt{enza.pellegrino@univaq.it}}, 
Laura Pezza$^\dag$, Francesca Pitolli\thanks{\noindent {\it Dip. SBAI, Universit\`{a} di Roma ''La Sapienza''} Via A. Scarpa 16, 00161 Roma, Italy. 
	e-mail: \tt{francesca.pitolli,laura.pezza@sbai.uniroma1.it} } }
	
\date{}

\begin{document}

\maketitle

\begin{abstract}
We used a collocation method in refinable spline space to solve a linear dynamical system having fractional derivative in time. The method takes advantage of an explicit derivation rule for the B-spline basis that allows us to efficiently evaluate the collocation matrices appearing in the method. We proof the convergence of the method. Some numerical results are shown.
\end{abstract}

\section{Introduction}
In recent years fractional differential models were used to describe a great variety of physical phenomena, such as the anomalous diffusion in biological tissues, the viscoelastic properties of smart materials, the growth of population in dynamical systems (see, \cite{KST 06,SaKiMa 93, Ta10} and references therein).
Even if there is a great effort in developing the theory of fractional calculus \cite{Diet,OlSp 74,Po99,SaKiMa 93}, the analytical solution of fractional differential problems can be obtained in a very few cases. This is why the literature on numerical methods to solve this kind of problems is growing rapidly (see, \cite{Ba10,LZ15} and the detailed bibliography in the more recent papers \cite{LC17,Pi18a}).
\\
In \cite{PePi 2018.a} two of the authors introduced an efficient {\em collocation method} to numerically solve fractional differential problems. 
In the method the approximating function is assumed to belong to a refinable space and its expression is determined by solving the differential problem in a set of collocation points. Thus, the method is both a projection method and a collocation method and the nonlocal behavior of the fractional derivative can easily be taken into account. 
\\
In the present paper we used this method to solve a linear dynamical system having fractional derivative in time. We assume the approximating function belongs to refinable spaces generated by the polynomial splines and we take advantage of an explicit derivation rule for the B-spline basis that allows us to efficiently evaluate the collocation matrix.
\\
The paper is organized as follows. In Section~\ref{Sec:2} we recall the definition of the Caputo fractional derivative and describe the fractional dynamical system we are interested in. We give also its analytical solution in terms of the matrix Mittag-Leffler function.
In Section~\ref{Sec:3} we describe the B-spline basis we used to construct the approximate solution to the differential system and give the explicit expression of the fractional derivatives of the basis functions. In Section~\ref{Sec:4} we analyze the refinability properties of the B-spline basis. The collocation method is described in Section~\ref{Sec:5} where its convergence is also proved. Finally, some numerical tests are provided in Section~\ref{Sec:6} while some conclusions are drawn in Section~\ref{Sec:7}. 
\section{Fractional dynamical systems}
\label{Sec:2}
Let $X(t): \RR \to \RR^m$ be a real-valued vector function, $X_0 \in \RR^m$ be a real vector and $A \in \RR^{m\times m}$ be a real matrix. We consider the following linear dynamical system: 
\begin{equation} \label{fracdynsystem}
\left \{
\begin{array}{ll}
D^{\gamma}_t X(t)=A \, X(t)\,, &  t>0\,, \quad 0<\gamma <1\,,\\ \\
X(0)=X_0\,,
\end{array}
\right.
\end{equation}
having time derivative of fractional, i.e. noninteger, order $\gamma$. 
In this context, the operator $D_t^\gamma$ denotes the {\em Caputo fractional derivative} with respect to the time $t$. For a sufficiently smooth vector function $X(t) = [x_1(t), x_2(t), \ldots, x_m(t)]^T$, the Caputo derivative is defined as 
\begin{equation} \label{Cap_fracder_vec}
D_t^\gamma X(t) := \bigl [D_t^\gamma x_1(t), D_t^\gamma x_2(t), \ldots, D_t^\gamma x_m(t) \bigr]^T\,,
\end{equation}
where
\begin{equation} \label{Cap_fracder}
D_t^{\gamma} \, x(t) := \bigl ({\cal J}^{(k-\gamma)} x^{(k)} \bigr)(t)\,, \quad k - 1 < \gamma < k\,, \quad k \in \NN\,, \quad t>0\,,
\end{equation}
and
\begin{equation} \label{RLint}
{\cal J}^{(\gamma)} x(t):=\frac 1{\Gamma(\gamma)}\,\int_0^t \, x(\tau)\,(t - \tau)^{\gamma-1} \, d\tau\,,
\end{equation}
is the {\em Riemann-Liouville integral operator}. Here,  
$
\Gamma(\gamma) 
$
denotes the Euler's gamma function. For details on fractional calculus see, for instance, \cite{Diet,KST 06, OlSp 74,SaKiMa 93}.

The existence of a unique solution to (\ref{fracdynsystem}) was proved, for instance, in \cite[\S7.1]{Diet}. A detailed analysis of positive linear systems of type (\ref{fracdynsystem}) and of their properties can be found in \cite{KR15} where the analytical solution in terms of the Mittag-Leffler function is obtained by the Laplace transform. Its explicit expression is
\begin{equation}
X(t)=E_{\gamma,1}(t^{\gamma},A)\,X_0\,, 
\end{equation}
where   
\begin{equation}
E_{\gamma,\beta}(z,A)=\sum_{k\ge0} \frac{(zA)^k}{\Gamma(\gamma k +\beta)}\,, \qquad z \in \CC\,, \quad A \in \RR^{m\times m}\,,
\end{equation}
is the matrix Mittag-Leffler function. We observe that the evaluation of $E_{\gamma,\beta}(z,A)$ is rather cumbersome (cf. \cite{GaPo 13}). An alternative expression of the analytical solution not involving the matrix Mittag-Leffler function can be found in \cite[\S7.1]{Diet}.

\section{The B-spline basis and its fractional derivatives}
\label{Sec:3}
In this section we describe the polynomial B-spline basis we will use to approximate the solution to Equation (\ref{fracdynsystem}) and give the analytical expression of its fractional derivative. 

The classical cardinal B-splines are piecewise polynomials of integer degree having breakpoints on integer knots (see \cite{dB07,Sch} for details). For our porposes, we define the cardinal B-spline through the truncated power function 
\begin{equation}  \label{trunc}
T_n(t) :=  \bigl(\max(0,t) \bigr) ^n, \quad n \in \NN\cup{0}\,,
\end{equation}
and the backward finite difference operator
\begin{equation} \label{ffdo} 
\Delta^n \, f(t):= \sum_{\ell=0}^{n} {\dbinom{n}{\ell}
	(-1)^l f(t-\ell)}\,, \quad n \in \NN\,.  
\end{equation}
Then, the cardinal B-spline of integer degree $n\ge 0$ is defined as
\begin{equation}  \label{B_spline}
B_n(t):=\frac{1}{n!}\,\Delta^{n+1} T_n (t), \qquad n \in \NN\cup{0}\,.
\end{equation}
The cardinal B-spline $B_n$ is a piecewise polynomial of degree $n$ with breakpoints on the integers, compactly supported on $[0,n+1]$ and belonging to $C^{n-1}$.
\\
On the semi-finite interval $[0,\infty)$ the integer translates
\begin{equation}
{\cal B}_n = \{B_n(t-\ell), -n \le \ell\}\,, \qquad t \in [0,\infty)\,,
\end{equation}
form a function basis for the spline space so that any spline function {\LARGE \calligra s} \, can be represented as
\begin{equation} \label{eq:rapp_f}
\mbox{\LARGE \calligra s}\,(t) = \sum_{\ell\ge -n} \, c_k \, B_n(t-\ell).
\end{equation}
As a consequence, the fractional derivative of {\LARGE \calligra s} \, can be evaluated as
\begin{equation} \label{eq:derfrac_f}
D^\gamma_t \mbox{\LARGE \calligra s}\,(t) = \sum_{\ell\ge -n} \, c_k \, D^\gamma B_n(t-\ell)\,, \qquad k-1<\gamma < k, \quad 1 \le k \le n-1.
\end{equation}
Thus, to compute the fractional derivatives of {\LARGE \calligra s} \, we need the fractional derivatives of the functions belonging to the B-spline basis $ {\cal B}_n$. 
\\
Let us denote by $B_{n,\ell}(t)$ the $\ell$-translate of $B_n$, i.e.
$$
B_{n,\ell}(t) = B_n(t-\ell)\,, \qquad \ell  \ge -n\,.
$$
First of all, we notice that when $\ell\ge 0$ the functions $B_{n,\ell}$ are {\em interior functions} having support $[\ell,\ell+n+1]$ all contained in $[0,\infty)$. Their fractional derivative can be evaluated by the differentiation rule 
\begin{equation}  \label{fracder_Bspline}
D^{\gamma}_t\, B_n(t) = \frac1{\Gamma(n-\gamma+1)} \, \Delta^{n+1} T_{n-\gamma} (t)\,,
\end{equation}
where
\begin{equation}  \label{eq:frac_trunc}
T_\gamma(t) :=  \bigl(\max(0,t) \bigr ) ^\gamma, \quad \gamma \ge 0,
\end{equation}
is the {\em fractional truncated power function} \cite{UB00}.
From (\ref{fracder_Bspline}) it follows that the fractional derivative of a polynomial B-spline is a fractional spline, i.e. a piecewice polynomial of noninteger degree. Details on fractional splines can be found in \cite{UB00}. 
\\
For $-n \le \ell \le -1$, the functions $B_{n,\ell}$ are {\em left edge functions} having support $[0,n+\ell+1]$. Their fractional derivative can be explicitly evaluated using definition (\ref{Cap_fracder}) and the differentiation rule (\ref{fracder_Bspline}) as the following theorem shows.

\begin{theorem} \label{Th:frac_der}
For $0< \gamma < 1$, the fractional derivative of the B-spline basis functions $B_{n,\ell}$ is given by
\begin{equation}\label{fracderivBint_0}
D_t^{\gamma}B_{n,\ell}(t) = \frac1{\Gamma(n-\gamma+1)} \, \Delta^{n+1} T_{n-\gamma} (t-\ell)\,, \qquad \ell \ge 0\,,
\end{equation}
and
\begin{equation}\label{fracderivB_left }
 D_{t}^{\gamma}B_{n,\ell}(t)=\frac{\Delta^{n+1}T_{n-\gamma}(t-\ell)}{\Gamma(n+1-\gamma)} - \frac{1}{\Gamma(1-\gamma)} \int_0^{-\ell} \frac{B'_n(\tau)}{(t-\ell-\tau)^{\gamma}} \, d\tau\,, \quad -n \le \ell \le -1\,.
\end{equation}
\end{theorem}

\begin{proof}	
The derivation rule (\ref{fracderivBint_0}) immediately follows from (\ref{fracder_Bspline}).
Now, consider the case $-n \le \ell \le -1$. From definition (\ref{Cap_fracder}) one has
\begin{eqnarray*} 
D_{t}^{\gamma}B_{n,\ell}(t) & = & \frac{1}{\Gamma(1-\gamma)}\int_0^t \frac{B'_{n,\ell}(\tau)}{(t-\tau)^{\gamma}}\, d\tau
= \\ \\
& = & \frac{1}{\Gamma(1-\gamma)} \left ( \int_{\ell}^t\frac{B'_{n,\ell}(\tau)}{(t-\tau)^{\gamma}} \, d\tau - \int_{\ell}^0\frac{B'_{n,\ell}(\tau)}{(t-\tau)^{\gamma}} \, d\tau \right)\,.
\end{eqnarray*}
The first integral is the fractional derivative of the $\ell$-translate of $B_n$ and can be evaluated by the differentiation rule (\ref{fracderivBint_0}).
\\
As for the second integral, we get
\begin{eqnarray*}
\frac{1}{\Gamma(1-\gamma)} \int_{\ell}^0\frac{B'_{n,\ell}(\tau)}{(t-\tau)^{\gamma}} \, d\tau = \frac{1}{\Gamma(1-\gamma)} \int_{\ell}^0\frac{B'_{n}(\tau-\ell)}{(t-\tau)^{\gamma}} \, d\tau = \\ \\
\rule{1cm}{0cm} \frac{1}{\Gamma(1-\gamma)} \int_0^{-\ell}\frac{B'_{n}(\tau)}{(t-\ell-\tau)^{\gamma}} \, d\tau
\end{eqnarray*}
so concluding the proof.
\end{proof} 
\noindent
In the following theorem we explicitly evaluate the integral appearing in the left hand side of (\ref{fracderivB_left }).

\begin{theorem}
For $-n \le \ell \le -1$ the explicitly expression of the integral in (\ref{fracderivB_left }) is
\begin{equation*}
\begin{array}{ll}
\displaystyle \frac{1}{\Gamma(1-\gamma)}\displaystyle \int_0^{-\ell}\frac{B'_n(\tau)}{(t-\ell-\tau)^{\gamma}}\,d\tau = \\ \\
\displaystyle \rule{1cm}{0cm} \frac{1}{\Gamma(n+1-\gamma)} \sum_{r=0}^{-\ell-1} (-1)^r \binom{n+1}{r}\bigg( (t-\ell-r)^{(n-\gamma)} + \\ \\
\displaystyle  \rule{1cm}{0cm} t^{1-\gamma}\, \sum_{p=0}^{n-1}  \frac{(-1)^{n-p}(-\ell-r)^{n-1-p}(t-\ell-r)^p}{(n-1-p)!}\prod _{s=1}^{n-1-p} (\gamma - s) \bigg ).
\end{array}
\end{equation*}
\end{theorem}
\begin{proof}
	We recall that $B_n'(t)$ writes:
	$$
	B'_n(t)=\frac{1}{n!}\Delta^{n+1} T'_n(t) = \frac{1}{(n-1)!}\sum_{r=0}^{n+1}(-1)^r\binom{n+1}{r}T_{n-1}(t-r)\,.
	$$
	Substituting the expression of $B'_n$ in the integral in the left hand side of (\ref{fracderivB_left }), we obtain
	$$
	\label{D_k0}
	\begin{array}{lcl}
	\displaystyle \frac{1}{\Gamma(1-\gamma)(n-1)!} 
	\sum_{r=0}^{n+1}(-1)^r\binom{n+1}{r}
	\int_{0}^{-\ell}\frac{T_{n-1}(\tau-r)}{(t-\ell-\tau)^{\gamma} } \, d\tau =\\ \\
	\displaystyle \frac{1}{\Gamma(1-\gamma)(n-1)!}
	\sum_{r=0}^{-\ell-1} (-1)^r\binom{n+1}{r} 
	\int_{0}^{-\ell-r}\frac{\tau^{n-1}}{(t-\ell-r-\tau)^{\gamma} } d\tau\,.
	\end{array}
	$$
	By a direct computation we get
	$$
	\begin{array}{l}
	\displaystyle
	\int_0^{-\ell}\frac{\tau^{n-1} }{(t-\ell-\tau)^{\gamma}}\,d\tau = \\ \\
	\displaystyle
	\frac{(n-1)!}{\prod_{s=1}^n(\gamma - s)}
	(t-\ell-\tau)^{(1-\gamma)} \sum_{p=0}^{n-1}\frac{(\ell-t)^{p}\tau^{n-1-p}}{(n-1-p)!}
	\prod_{s=1}^{n-1-p} (\gamma - s)\bigg|_{\tau=0}^{-\ell} 
	= \\ \\
	\displaystyle \frac{(n-1)!}{\prod_{s=1}^n(\gamma - s)} \left[(t-\ell)^{(n-\gamma)}+t^{1-\gamma}\,  \sum_{p=0}^{n-1}\frac{(-1)^{n-p}(-\ell)^{n-1-p}(t-\ell)^p}{(n-1-p)!} \prod_{s=1}^{n-1-p} (\gamma - s) \right ]
	\end{array}
	$$
	and the claim follows.
\end{proof}


\section{Multiresolution Analysis on $L_2[0,\infty)$}
\label{Sec:4}
The B-spline basis ${\cal B}_n$ generates a multiresolution analysis on the semi-infinite interval $[0,+\infty)$ \cite{Chui}. This means that  the sequence of subspaces $\{ V_j\subset L_2[0,+\infty)\}$ defined as
\begin{equation*}
V_j = {\rm clos}_{L_2[0,\infty)} \,\{\phi_{j\ell}(t)\,, \ell \ge -n \} \,, \qquad j \in \ZZ\,, \qquad t \ge 0\,,
\end{equation*}
where 
\begin{equation*} \label{phi_jk}
\phi_{j\ell}(t) := 2^{j/2} \, B_n(2^j\,t -\ell)\,,
\end{equation*}
fulfills the following properties:

\medskip
\begin{tabular}{rlrl}
	({\it i}) & $V_j \subset V_{j+1}$, $j \in \ZZ$;  &
	({\it ii}) & $\overline{\cup_{j \in \ZZ} V_j} = L^2[0,+\infty)$; \\ \\
	({\it iii})& $\bigcap_{j\in \ZZ} V_j = \{0\}$; &
	({\it iv}) & $f(t) \in V_j \leftrightarrow f(2t) \in V_{j+1}$, $j\in \ZZ$; \\ \\
	({\it v}) & there exists a & & \hskip -2.cm $L_2[0,+\infty)$-stable basis in $V_0$.
\end{tabular}
\\ \\
Thus, any function $f \in V_j$ can be represented as
\begin{equation}
f(t) = \sum_{\ell \ge -n} \, c_{j\ell} \, \phi_{j\ell}(t)\,,
\end{equation}
where $\{c_{j\ell}\}\in \ell_2(\ZZ)$.
Once again, the basis functions $\phi_{j\ell}$ with $-n \le \ell \le -1$ are the left egde functions, while for $\ell \ge 0$  $\phi_{j\ell}$ are integer translates of  $B_n(2^j\cdot)$. The computation of the fractional derivatives of $f$ requires the evaluation of the fractional derivatives of $\phi_{j\ell}$. This can be done using Theorem \ref{Th:frac_der} and the following lemma. 

\begin{lemma}
	The Caputo derivative of order $\gamma$ of the $2^j$-dilate of a function $f(t)$ is given by:
	$$
	{ D^{\gamma}_t f(2^j t)=2^{\gamma j} D_{2^j t}^{\gamma}f(2^j t)}\,, \quad k-1 < \gamma < k\,, \quad k \in \NN\,.
	$$
\end{lemma}
\begin{proof}
	Let $F(t)=f(2^j t)$, then  $F^{(m)}(t)=2^{jm}f^{(m)}(2^j t)$, $m \in \NN$. By definition
	$$ \begin{array}{lcl}
	D^{\gamma}_tF(t)&=&\displaystyle \frac{1}{\Gamma (k-\gamma)}\int_0^{t}\frac{F^{(k)}(\tau)}{( t-\tau)^{\gamma - k +1}} d\tau \\ \\
	&=&\displaystyle \frac{1}{\Gamma (k-\gamma)}\, 2^{jk}\int_0^{t}\frac{ f^{(k)}(2^j\tau)}{( t-\tau)^{\gamma - k +1}} d\tau
	\end{array}
	$$
	By the change of variables  $2^j \tau \rightarrow \tau$, we get
	$$ \begin{array}{lcl}
	D^{\gamma}_t F(t)&=&\displaystyle\frac{1}{\Gamma (k-\gamma)}2^{j(k-1)}\int_0^{2^j t}\frac{ f^{(k)}( \tau)}{( t-2^{-j} \tau)^{(\gamma - k +1)}} \,d \tau\\ \\
	&=&\displaystyle\frac{1}{\Gamma (k-\gamma)} \frac{2^{j(k-1)}}{2^{-j(\gamma-k+1)}}\int_0^{2^j t}\frac{ f^{(k)}( \tau)}{(2^{j} t- \tau)^{(\gamma - k +1)}} 
	d \tau 
	\end{array}
	$$
	and the claim follows.
\end{proof}
In the next section we will describe how to apply the {\it collocation method} introduced in \cite{PePi 2018.a} to numerically solve the differential problem (\ref{fracdynsystem}) in the refinable spline spaces.



\section{\bf The fractional collocation method}
\label{Sec:5}
\noindent
We look for an approximating vector function 
\begin{equation} \label{xj}
X_j(t) = \sum_{\ell\ge -n} \, C_{j\ell} \, \phi_{j\ell}(t) \in V_j\,, \quad C_{j\ell} \in \RR^m,
\end{equation}
that solves the differential problem (\ref{fracdynsystem}) on a set of {\em  collocation points}.
We choose as collocation points the dyadic nodes in which $\phi_{jk}$ can be efficiently evaluated through well-known recursive algorithms \cite{Ma99}.
\\
Let ${\cal I} = [0,T]$ be a finite interval. Without loss of generality we assume $T\in \NN$. Since $\phi_{jk}$ has compact support, for $t \in {\cal I}$ the sum in (\ref{xj}) reduces to a finite sum:
\begin{equation} \label{xjI}
X_j(t) = \sum_{\ell= -n}^{2^jT-1} \, C_{j\ell} \, \phi_{j\ell}(t), \quad t \in {\cal I}.
\end{equation}
Let us denote by $\{t_p =  p/2^s, 0\le p \le 2^s T\}$ the dyadic collocation points in the interval $\cal I$.
Substituting (\ref{xjI}) in (\ref{fracdynsystem}) evaluated on the collocation points gives
\begin{equation} \label{collsyst}
\left \{
\begin{array}{ll}
D^{\gamma}_t X_j(t_p)=A \, X_j(t_p), &  1\le p \le 2^s T, \\ \\
X_j(0)=X_0.
\end{array}
\right.
\end{equation}
This is a linear algebraic system that can be written in matrix form as follows 
\begin{equation} \label{matrix_coll}
\left \{
\begin{array}{ll}
(I_m\otimes G_{js}-A\otimes B_{js} )\, \Gamma_{js} = 0\,,\\ \\
I_m\otimes\Phi_{js}(0) \, \Gamma_{js}=X_{0}\,,
\end{array}
\right.   
\end{equation}
where  $I_m$ is the identity matrix of order $m$,
$$
G_{js}= \bigl [D^{\gamma}_t \, \phi_{j\ell}(t_p), 0< p\le 2^sT, -n \le \ell \le 2^jT-1 \bigr]
$$
and
$$ 
B_{js}= \bigl [\phi_{j\ell}(t_p), 0< p\le 2^sT, -n \le \ell \le 2^jT-1 \bigr] 
$$
are the collocation matrices of the refinable basis, 
$$
\Phi_{js}(0)=\bigl[\phi_{j\ell}(0),-n \le \ell \le 2^jT-1 \bigr]\,,
$$
and
$$
\Gamma_{js}=\bigl[C_{j\ell},-n \le \ell \le 2^jT-1 \bigr]^T
$$
is the unknown vector. We notice that the linear system (\ref{matrix_coll}) has $m(2^sT+1)$ equations and $m(2^jT+n)$ unknowns. To guarantee the existence of a unique solution we set $2^sT+1 \ge 2^jT+n$.  

\begin{theorem}
	For $2^sT+1 \ge 2^jT+n$ the linear system (\ref{matrix_coll}) has a unique solution. 
\end{theorem}

\begin{proof}
	Using definitions (\ref{Cap_fracder_vec})-(\ref{RLint}) the differential problem (\ref{fracdynsystem}) can be written as a system of integral equations:
	\begin{equation} \label{int_eq}
	Z(t) = A \, J^{(\gamma)} \, Z(t) + X_0,
	\end{equation}
	where $Z(t) = [z_1(t)=D^\gamma_t x_1(t),z_2(t)=D^\gamma_t x_2(t),\ldots,z_m(t)=D^\gamma_t x_m(t)]^T$ and $J^{(\gamma)} \, Z = [J^{(\gamma)} \, z_1(t),\ldots,J^{(\gamma)}z_m(t)]^T$.
	The system above is equivalent to the differential problem (\ref{fracdynsystem}) (cf. \cite{KPT16}) and has a unique solution \cite{Va93} so that the associated integral operator is invertible. Thus, the linear system (\ref{matrix_coll}) has a unique solution, too (cf. \cite{As78}).
\end{proof}
\noindent
Finally, we proof the convergence of the collocation method (\ref{collsyst}).

\begin{theorem}
The collocation method is convergent, i.e. 
$$
\lim_{j \to \infty}\| X(t) - X_j(t)\|_\infty = 0,
$$
where $\|X(t)\|_\infty = \max_{1\le i \le m} \left (\max_{t \in [0,T]}|x_i(t)| \right)$. Moreover, the approximation order is $\gamma$, i.e.
$$
\| X(t) - X_j(t)\|_\infty \le \kappa \, 2^{-j\gamma}, \qquad 0 < \gamma <1,
$$
where $\kappa$ is a constant independent from $j$.
\end{theorem}

\begin{proof}
	Since the collocation method can be used also to approximate the solution to the system (\ref{int_eq}), the equivalence implies that the approximation error $\| X - X_j\|_\infty$ is the same as the approximation error $\|Z-Z_j\|_\infty$ \cite{As78,KPT16}. Now, $Z_j$ is a projection operator in the spline space so that the convergence is guarantee with approximation order at least $\gamma$ in the case when $Z$ is sufficiently smooth \cite{dB07}.
\end{proof}

We notice that since the equality $2^sT+1=2^jT+n$ can be satisfied just in a few special cases, in practice we choose $s$ and $j$ so that $2^sT+1>2^jT+n$. Thus, the system (\ref{matrix_coll}) results is an overdetermined linear system that can be solved in the least squares sense.

\section{Numerical results}
\label{Sec:6}
In this section we use the proposed method to solve some test problems. 
In the tests we used the splines of degree $n=3$ and $n=4$ as approximating functions and set $s=j+1$. The functions of the cubic B-spline basis and their fractional derivatives are shown in Figures~\ref{Fig:Basis1}-\ref{Fig:Basis2}. The ordinary first derivative is also displayed. 
\\
To check the accuracy of the approximations obtained by the proposed method, we evaluated the componentwise $L_\infty$-norm of the error ${\cal E}_j(t)=X(t)-X_j(t)$, i.e.
$$
\|{\cal E}_{i,j}(t)\|_\infty = \max_{t \in [0,T]}|x_i(t)-x_{i,j}(t)|\,, \qquad 1 \le i \le m\,.
$$
Moreover, we evaluated the numerical approximation order $\rho_{\gamma,n}(j)$ defined as
$$
\rho_{\gamma,n}(j) = \log \left ( \frac {\|{\cal E}_{i,j}(t)\|_\infty}{\|{\cal E}_{i,j+1}(t)\|_\infty} \right) \frac 1{\log(2)} \,.
$$
We notice that the matrix Mittag-Leffler function appearing in the analytical solution was evaluated using the procedure proposed in \cite{GaPo 13}.

\begin{figure}[t]
	\hskip .8cm
	\begin{tabular}{cc}
		\includegraphics[width=6cm]{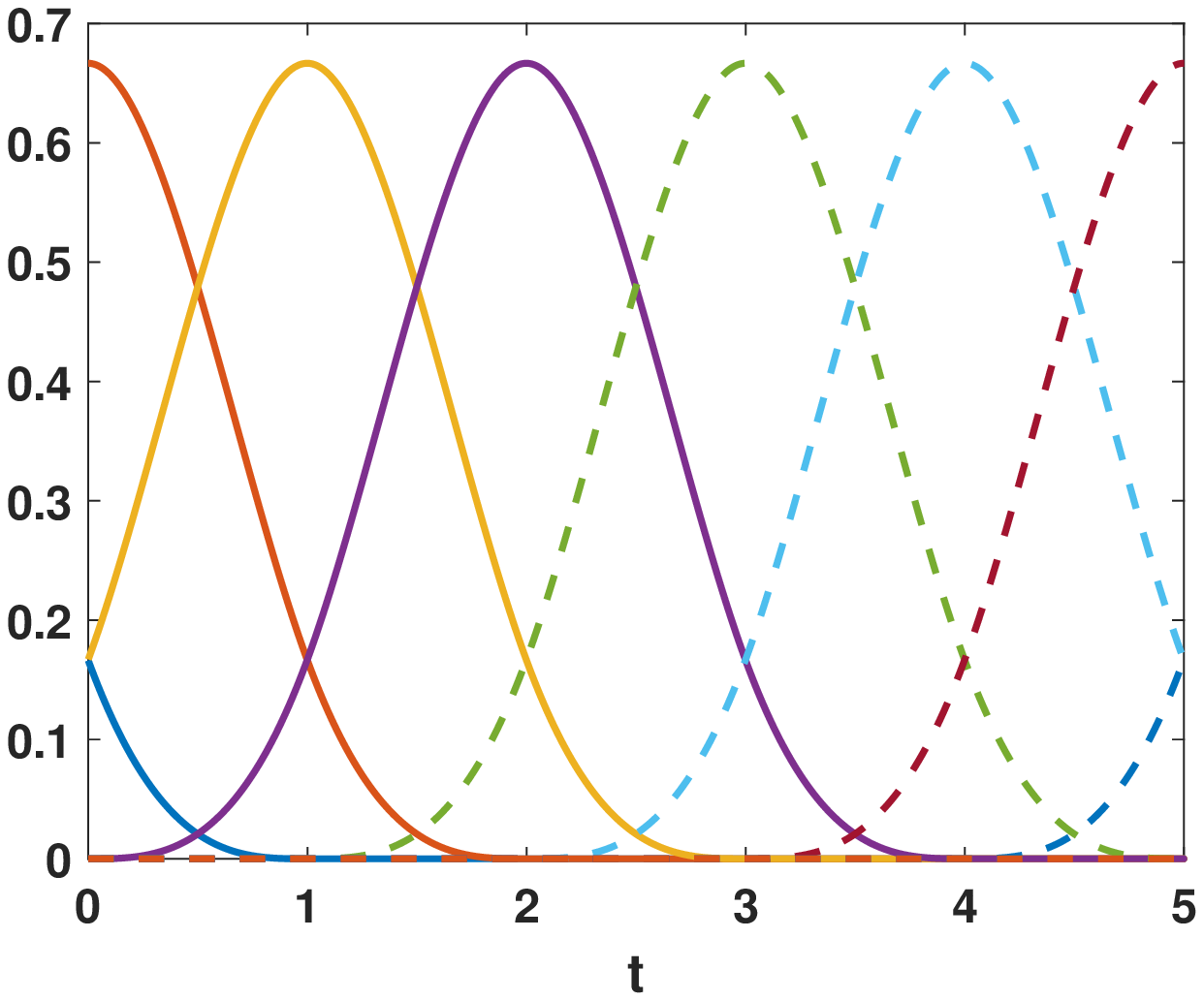}
		& 
		\includegraphics[width=6cm]{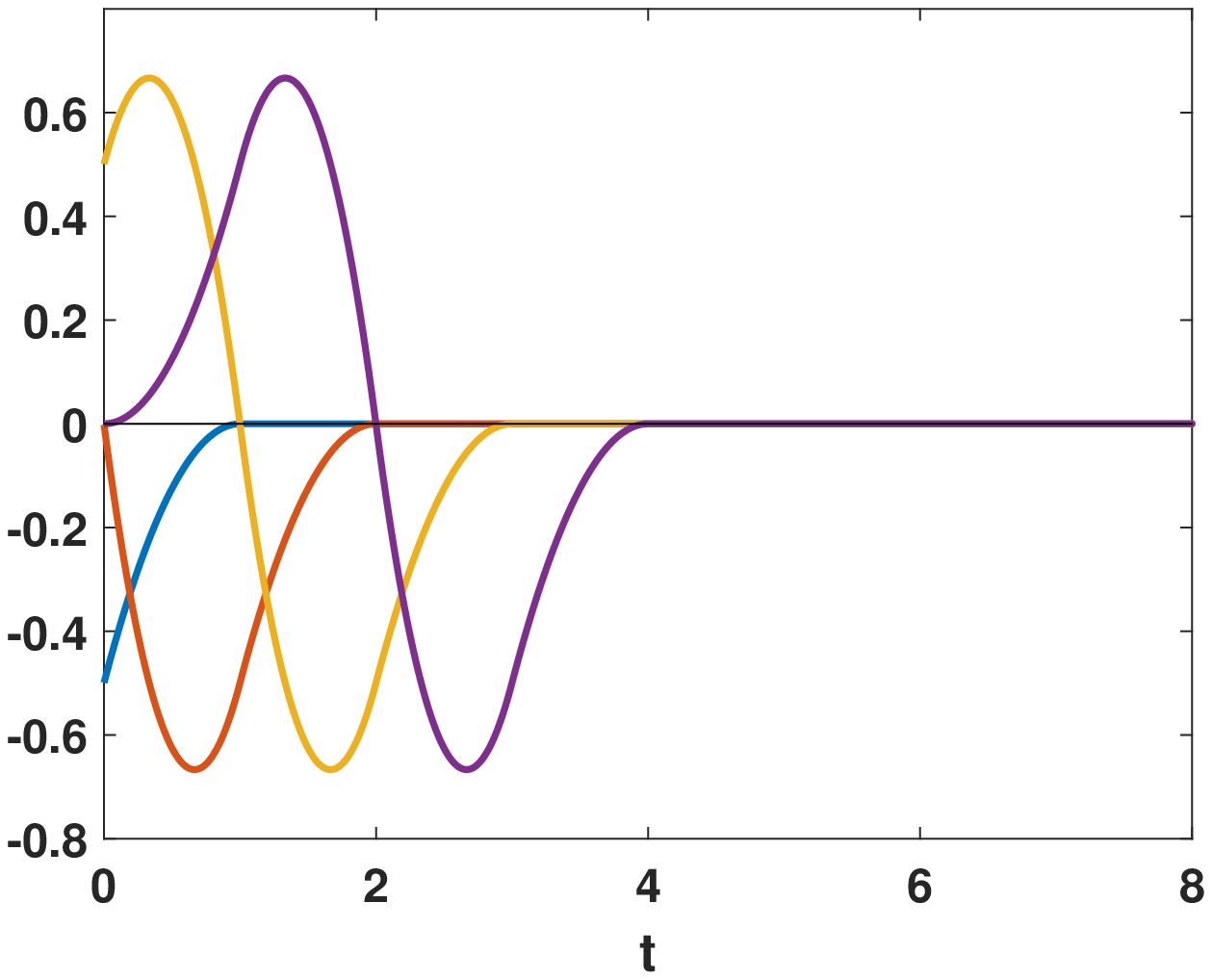}
	\end{tabular}
	\caption{The cubic B-spline basis (left panel) and the ordinary first derivative of the first four basis functions (right panel). The three boundary functions and the first interior function are displayed as solid lines.}
	\label{Fig:Basis1}
\end{figure}
\begin{figure}[h]
	\hskip .8cm
	\begin{tabular}{cc}
		\includegraphics[width=6cm]{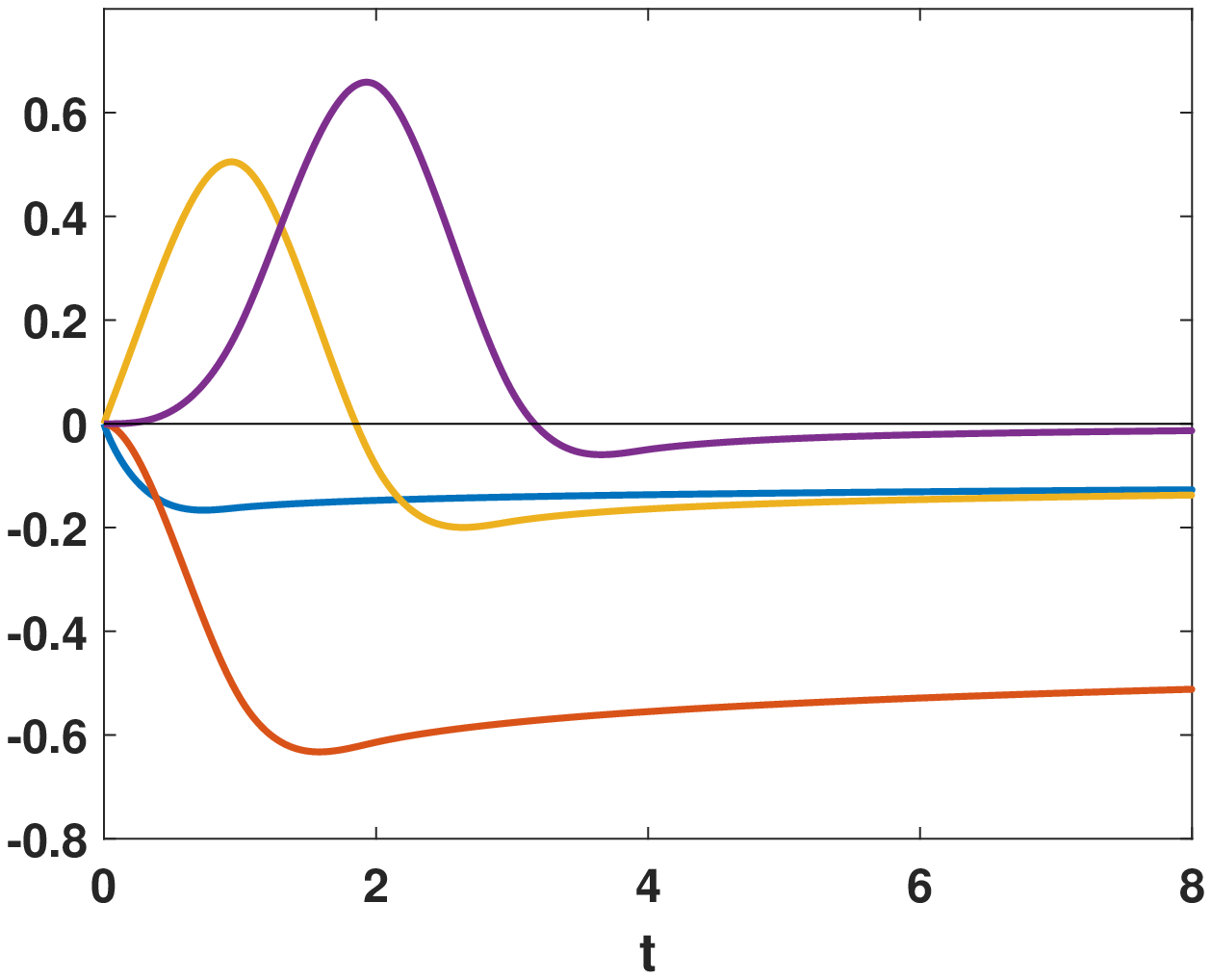}
		& 
		\includegraphics[width=6cm]{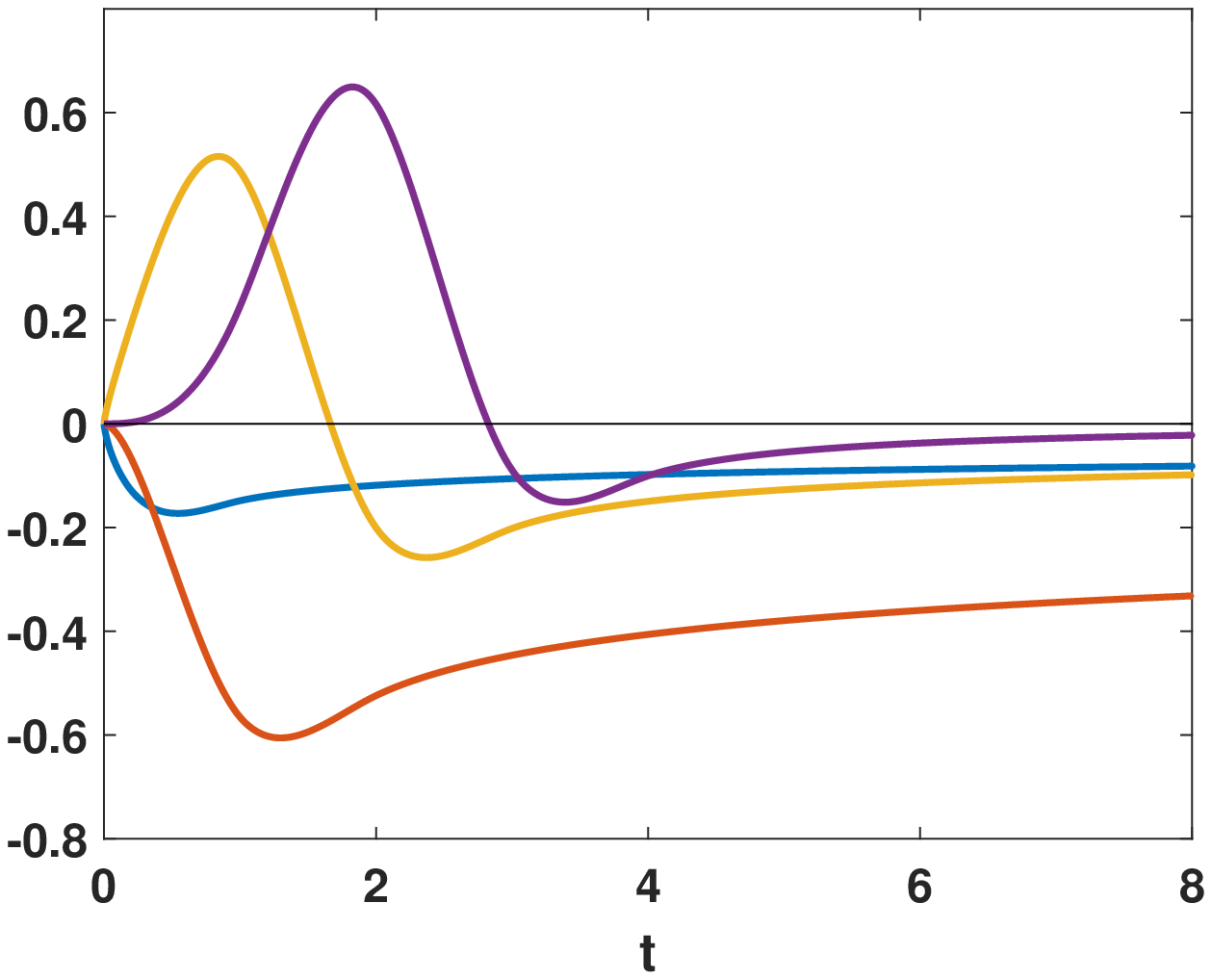}
		\\
		\includegraphics[width=6cm]{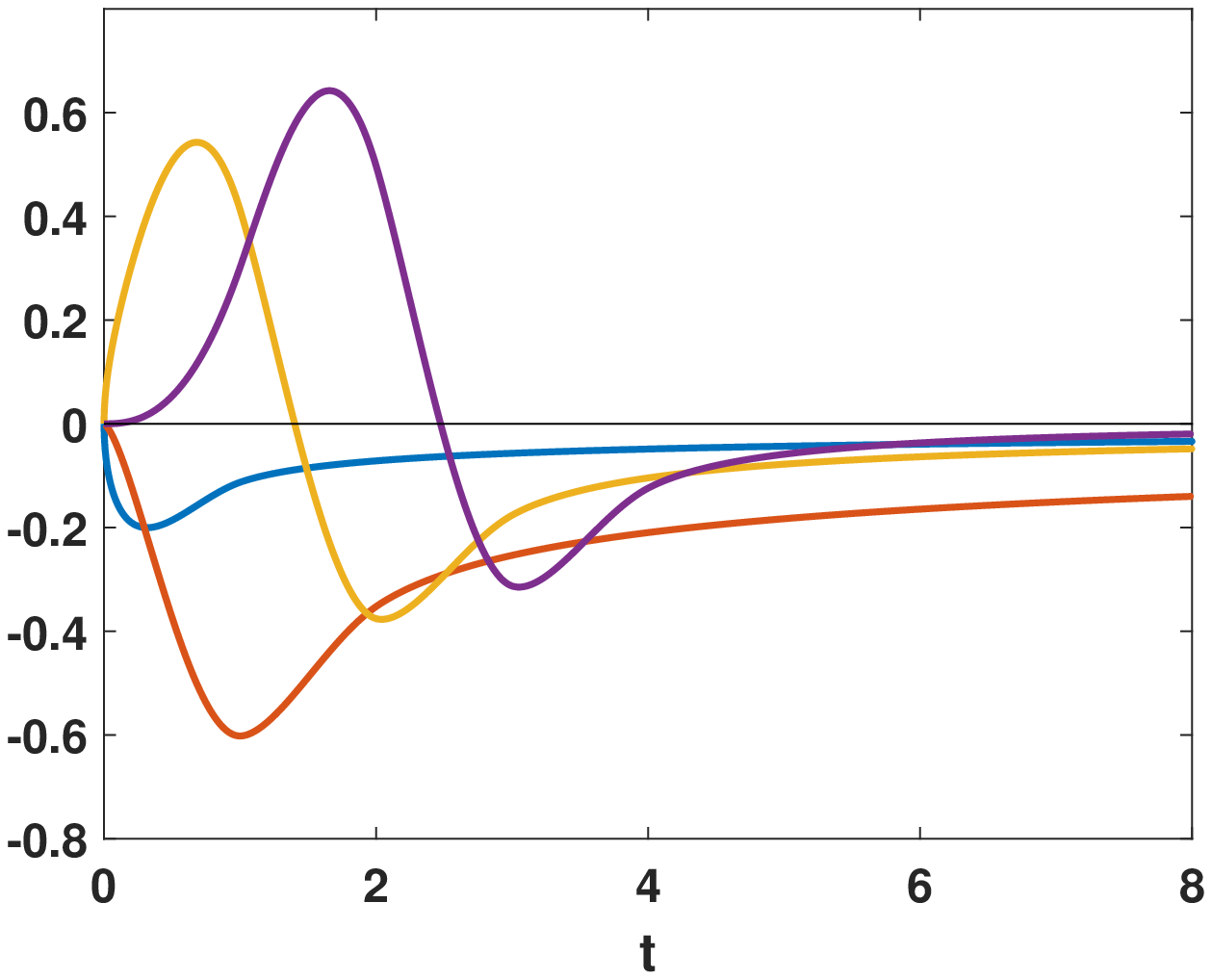}
		& 
		\includegraphics[width=6cm]{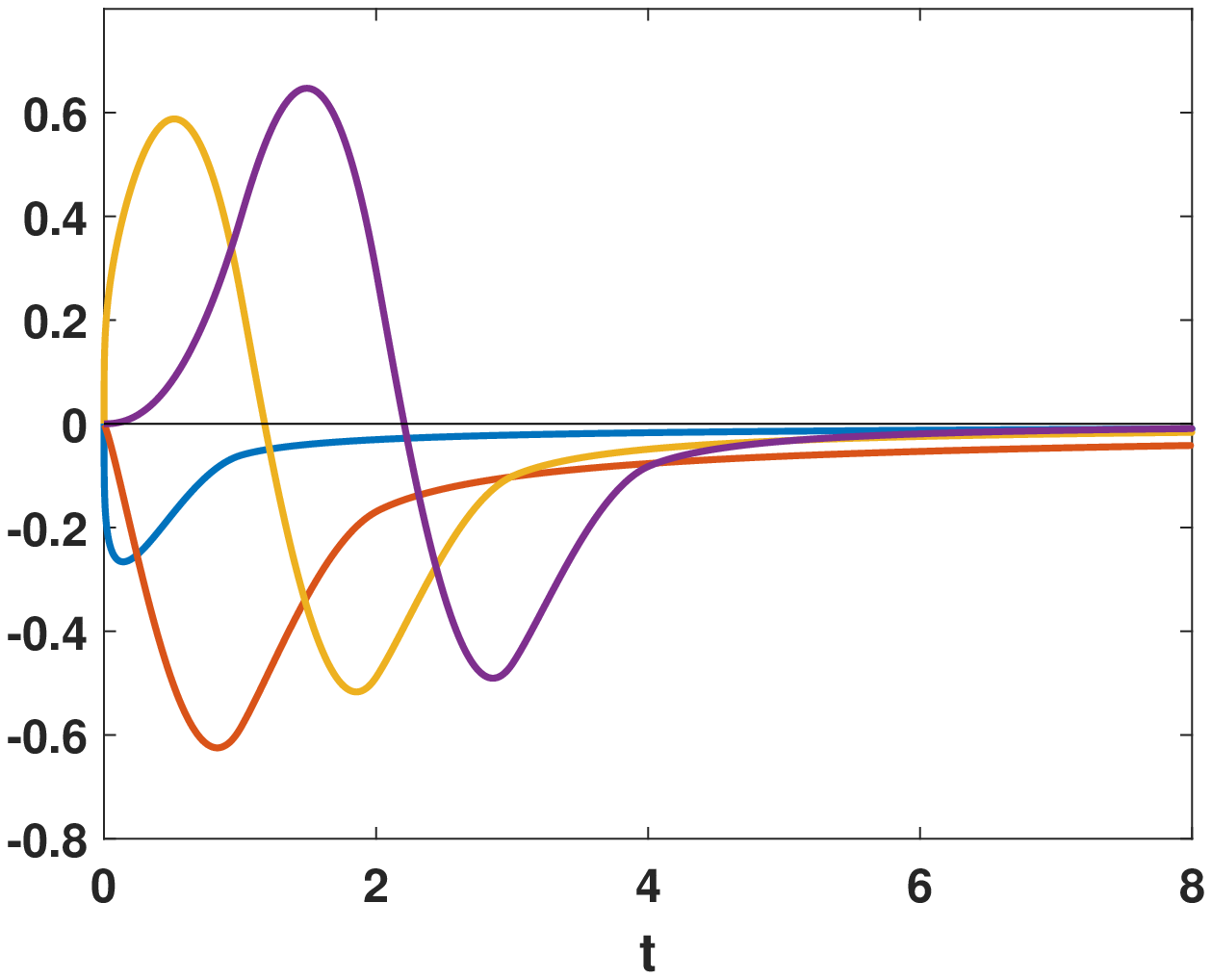}
	\end{tabular}
	\caption{The fractional derivative of the boundary functions and of the first interior function for $\gamma=0.10$ (left top panel), 0.25 (right top panel), 0.50 (left bottom panel), 0.75 (right bottom panel).}
	\label{Fig:Basis2}
\end{figure}

\subsection{Example 1}
First of all, we tested the accuracy of the collocation method by solving the following simple fractional differential equation (cf. \cite[pg. 137]{Diet}):
\begin{equation} \label{Num1}
 \left \{ \begin{array}{ll}
\displaystyle D^{\gamma} x(t)=- x(t)+t^2+2\,\frac{t^{2-\gamma}}{\Gamma(3-\gamma)}, \qquad t>0, \quad 0<\gamma<1, \\ 
x(0)= 0.     
\end{array} \right.  
\end{equation}
whose exact solution is $x(t)=t^2$. In this case the cubic spline approximation is exact.  
We numerically solve Equation~(\ref{Num1}) in the interval ${\cal I} = [0,1]$ for $\gamma=0.10$, 0.25, 0.50, 0.75. The table below lists the $L_\infty$-norm of the error ${\cal E}_j(t) = x(t)-x_j(t)$ obtained by the collocation method when $j=7$:
\medskip
\begin{center}
\begin{tabular}{c|c}
	$\gamma$ & $\|x-x_j\|_\infty$\\
	\hline
	0.10 &2.15e-16\\
	0.25 &3.16e-16\\
	0.50 &3.77e-16\\
	0.75 &6.42e-16
\end{tabular}
\end{center}
\medskip
As expected, the error is in the order of the machine precision. 

\subsection{Example 2}
In the second test we solved the fractional dynamical system 
\begin{equation} \label{Num2}
\left \{ \begin{array}{ll}
D^{\gamma} x(t)=-\frac32 x(t) +\frac12  y(t), \\ 
& \qquad t>0\,, \qquad 0 < \gamma < 1\,.\\
D^{\gamma} y(t)= \frac12 x(t)-\frac32y(t), \\ \\
x(0)=1\,, \quad y(0)= 2\,.
\end{array} \right.  
\end{equation}
The exact solution is \cite[\S7.1]{Diet}
$$
\begin{array}{l} 
x(t) = \frac32 \, E_{\gamma}(-t^\gamma) -\frac12 E_{\gamma}(-2t^\gamma)\,, \\ \\
y(t) = \frac32 \, E_{\gamma}(-t^\gamma) + \frac12\,E_{\gamma}(-2t^\gamma)\,,
\end{array} 
$$
where 
$$
E_\gamma(t) = \sum_{k\ge0} \frac{t^k}{\Gamma(\gamma k +1)}\,, 
$$
is the one-parameter Mittag-Leffler function.
We notice that the matrix 
$$A = \frac12\left [ \begin{array}{cc}
-3 & 1 \\
1 & -3
\end{array} \right]$$ 
associated with the dynamical system (\ref{Num2}) has negative eigenvalues, so that the stability of the dynamical system is guaranteed \cite{KR15}.\\
We solved the differential problem (\ref{Num2}) by the collocation method described in Section~\ref{Sec:5} for $\gamma=0.10$, 0.25, 0.50, 0.75, and for different values of $j$.  
In Figures~\ref{Fig:ex2gamma0p1}-\ref{Fig:ex2gamma0p75} the numerical solution and the approximation error are displayed in the case of the cubic spline approximation and for $j=8$. The numerical solution and the error obtained when solving the classical problem with integer first derivative are displayed in Figure~\ref{Fig:ex2gamma1}. The plots show that the proposed method gives a good accuracy that increases as $\gamma$ increases, i.e. as the smoothness of the analytical solution increases.
In Figure~\ref{Fig:ordine_gamma} the numerical approximation order $\rho_{\gamma,n}(j)$ is displayed as a function of $j$ for different values of $\gamma$ and $n=3$ and $n=4$. The plots show that the numerical approximation order is in accordance with the theoretical one. Moreover, the error is lower for the spline of degree 4.
Finally, in Figure~\ref{Fig:ordine_1} the numerical convergence order $\rho_{\gamma,n}(j)$ for $n=3$ and $n=4$ is displayed in the case of ordinary first derivative. We observe that in this case the theoretical convergence order is $n+1$ (cf. \cite{dB07}).
 
\begin{figure}[ht]
	\centering
	\includegraphics[width=12cm]{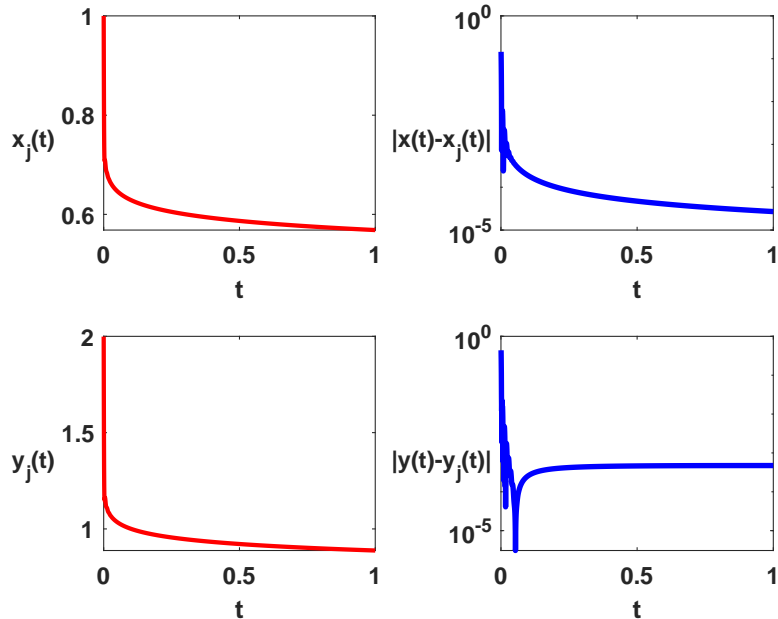}
	\caption{The numerical solutions  $x_{j}, \,y_{j}$, for $j=8$ (left panels) and the approximation error (right panels) obtained with the cubic spline when $\gamma=0.10$.}
	\label{Fig:ex2gamma0p1}
\end{figure}

\begin{figure}[ht]
	\centering
	\includegraphics[width=12cm]{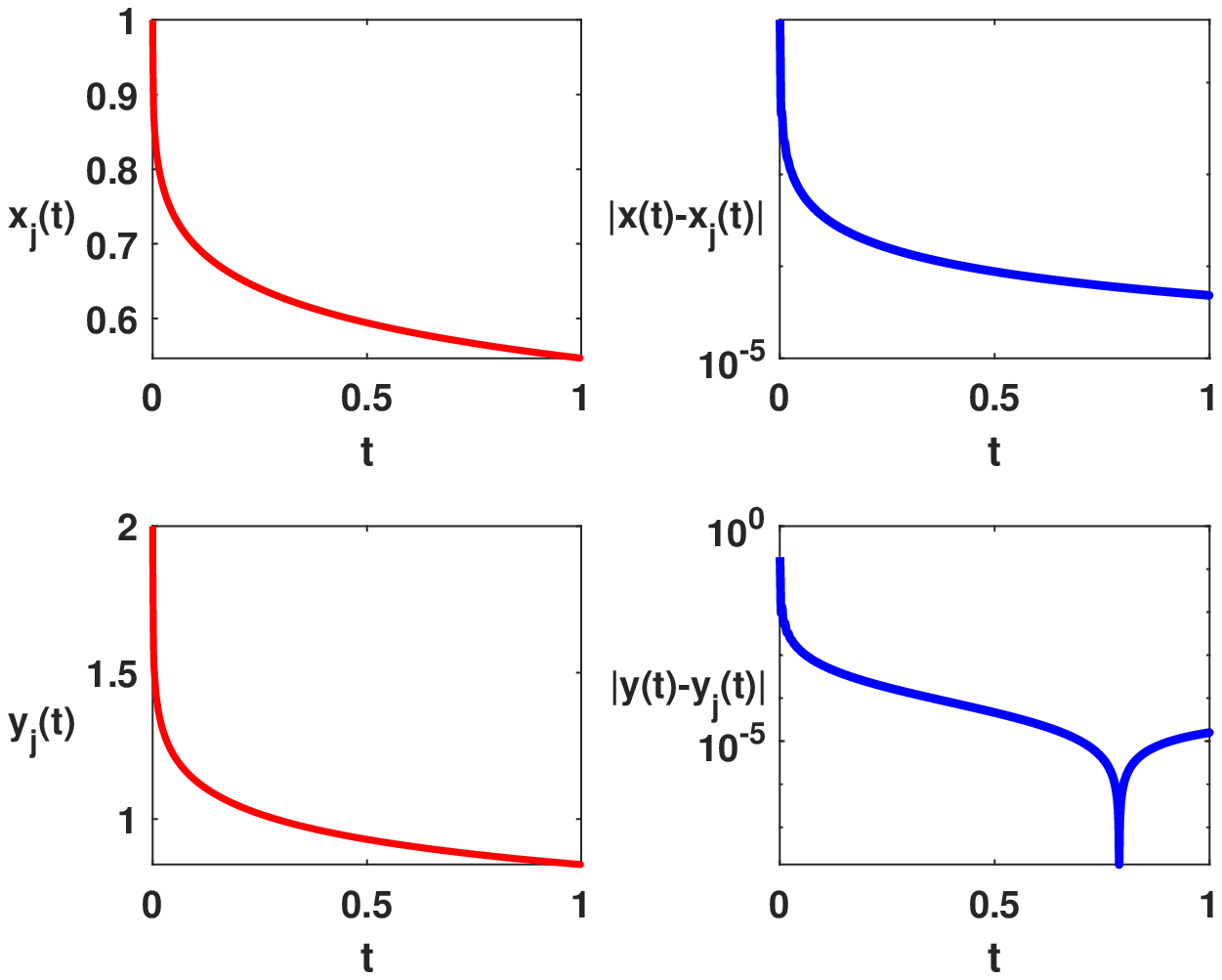}
	\caption{The numerical solutions  $x_{j}, \,y_{j}$, for $j=8$ (left panels) and the approximation error (right panels) obtained with the cubic spline when $\gamma=0.25$.}
	\label{Fig:ex2gamma0p25}
\end{figure}

\begin{figure} [ht]
	\centering
	\includegraphics[width=12cm]{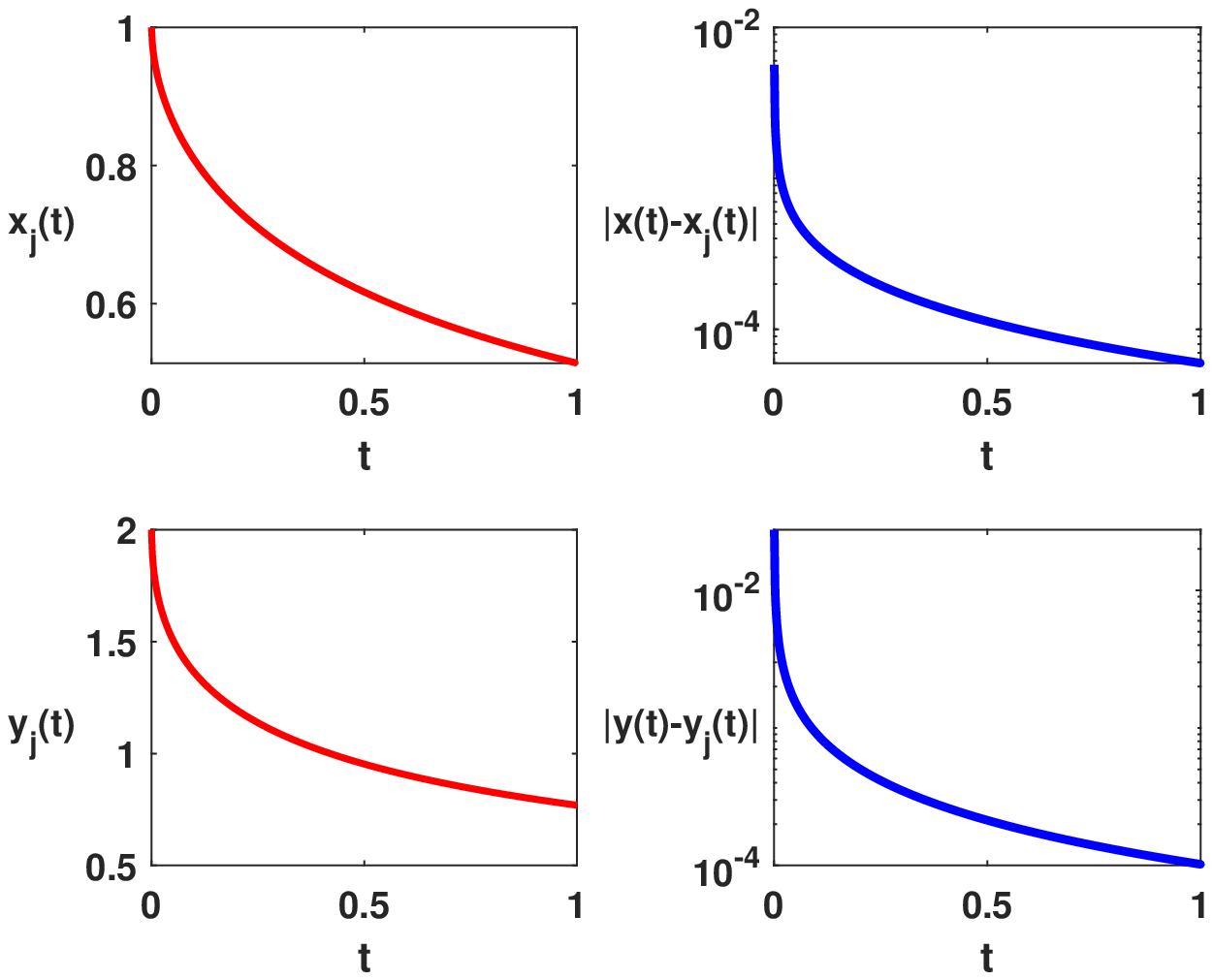} 
	\caption{The numerical solutions  $x_{j}, \,y_{j}$, for $j=8$ (left panels) and the approximation error (right panels) obtained with the cubic spline when $\gamma=0.50$.}
	\label{Fig:ex2gamma0p5}
\end{figure}

\begin{figure} [ht]
	\centering
	\includegraphics[width=12cm]{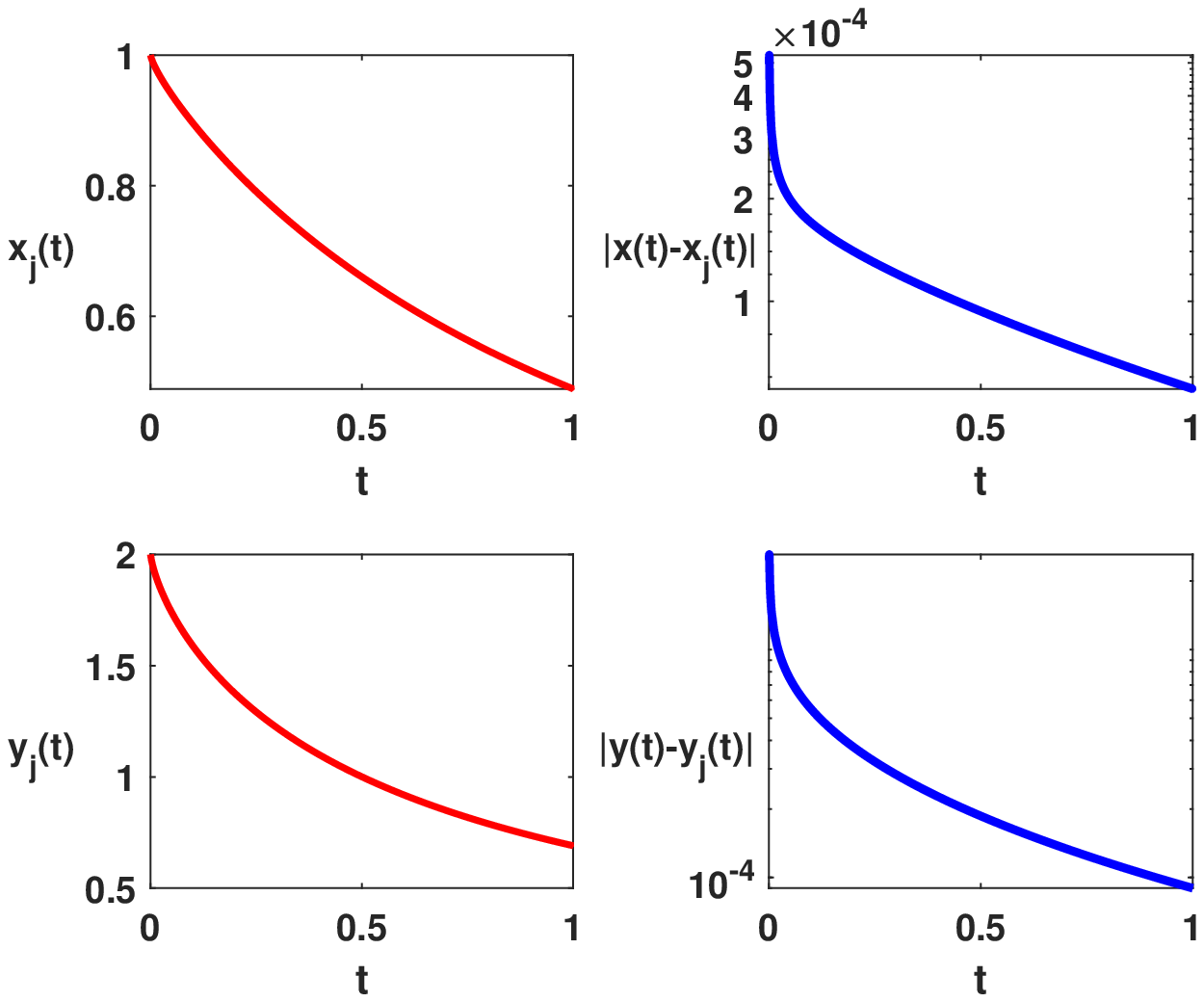}
	\caption{The numerical solutions  $x_{j}, \,y_{j}$, for $j=8$ (left panels) and the approximation error (right panels) obtained with the cubic spline when $\gamma=0.75$.}
	\label{Fig:ex2gamma0p75}
\end{figure}

\begin{figure} [ht]
	\centering
	\includegraphics[width=12cm]{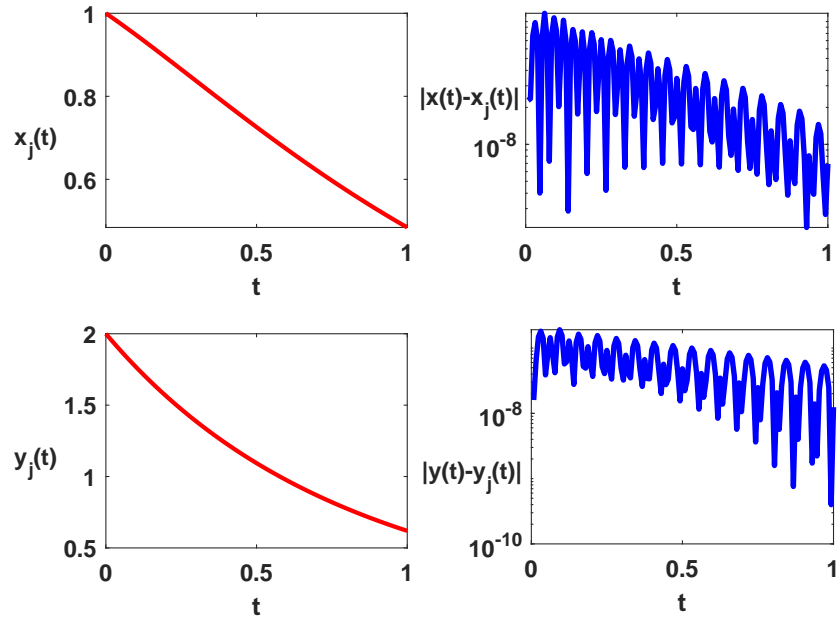}
	\caption{The numerical solutions  $x_{j}, \,y_{j}$, for $j=4$ (left panels) and the approximation error (right panels) obtained with the cubic spline for the classical problem having ordinary first derivative.}
	\label{Fig:ex2gamma1}
\end{figure}

\begin{figure} [ht]
	\centering
	\begin{tabular}{cc}
	\includegraphics[width=5cm]{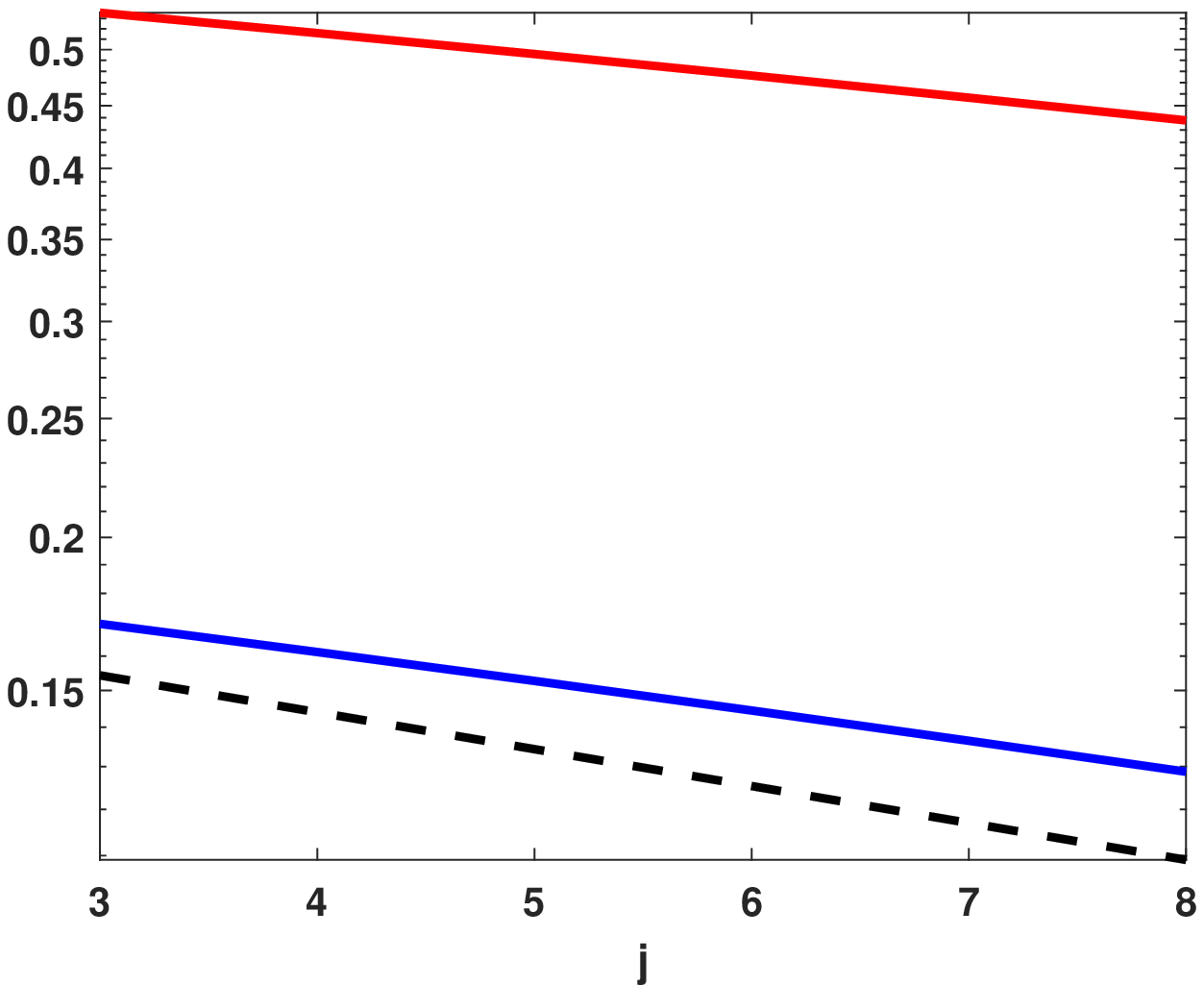} & \includegraphics[width=5cm]{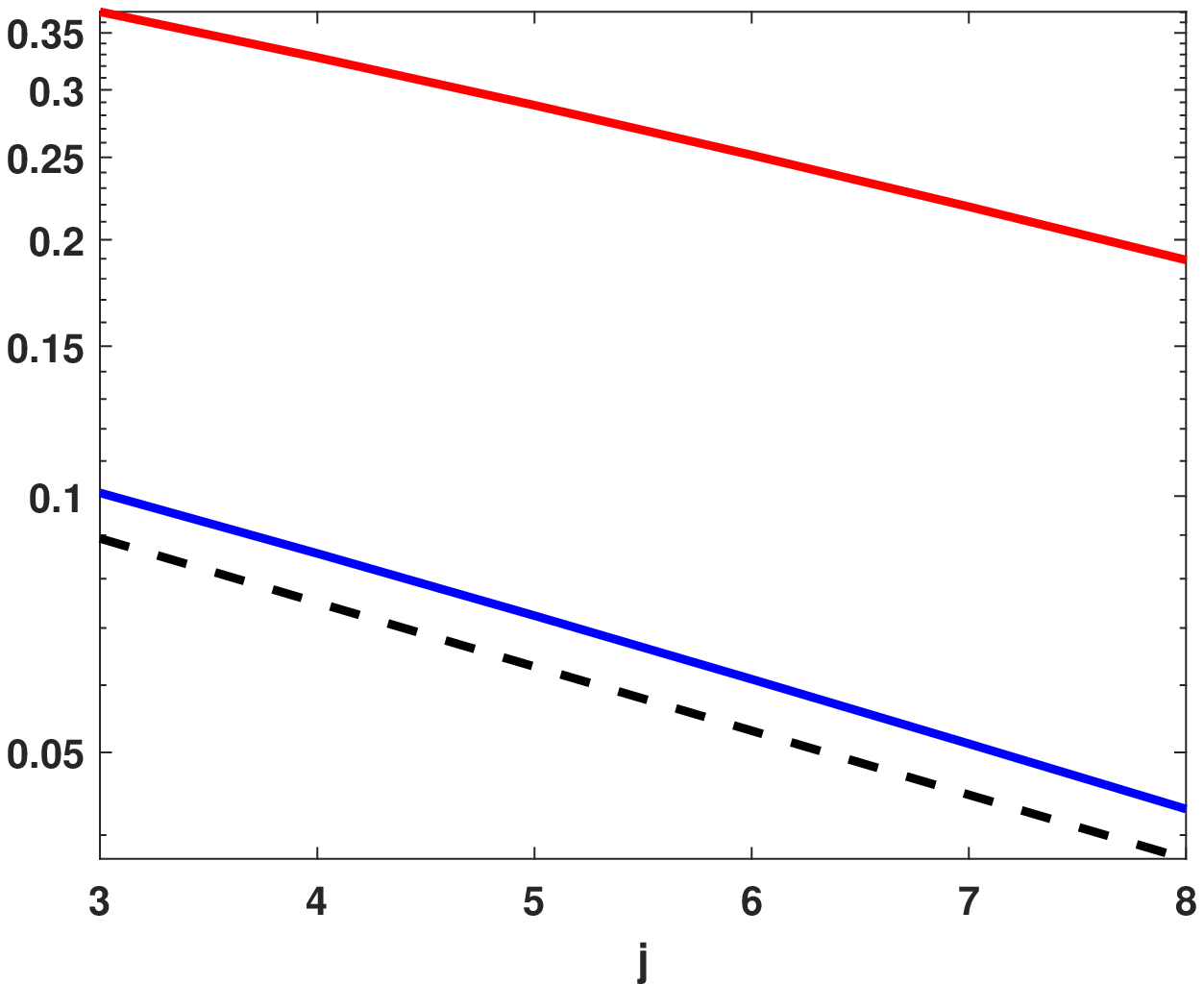} \\
	\includegraphics[width=5cm]{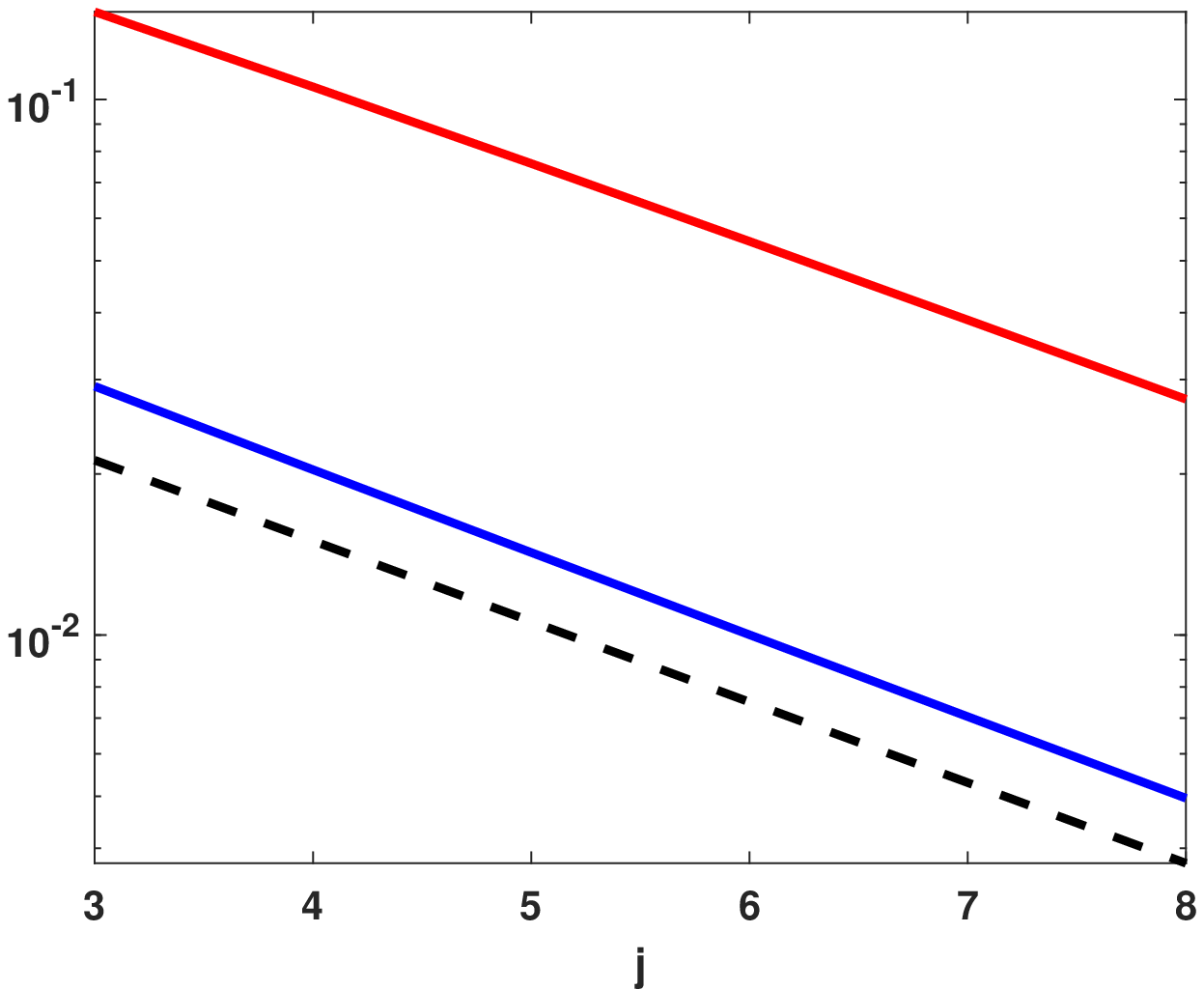} & \includegraphics[width=5cm]{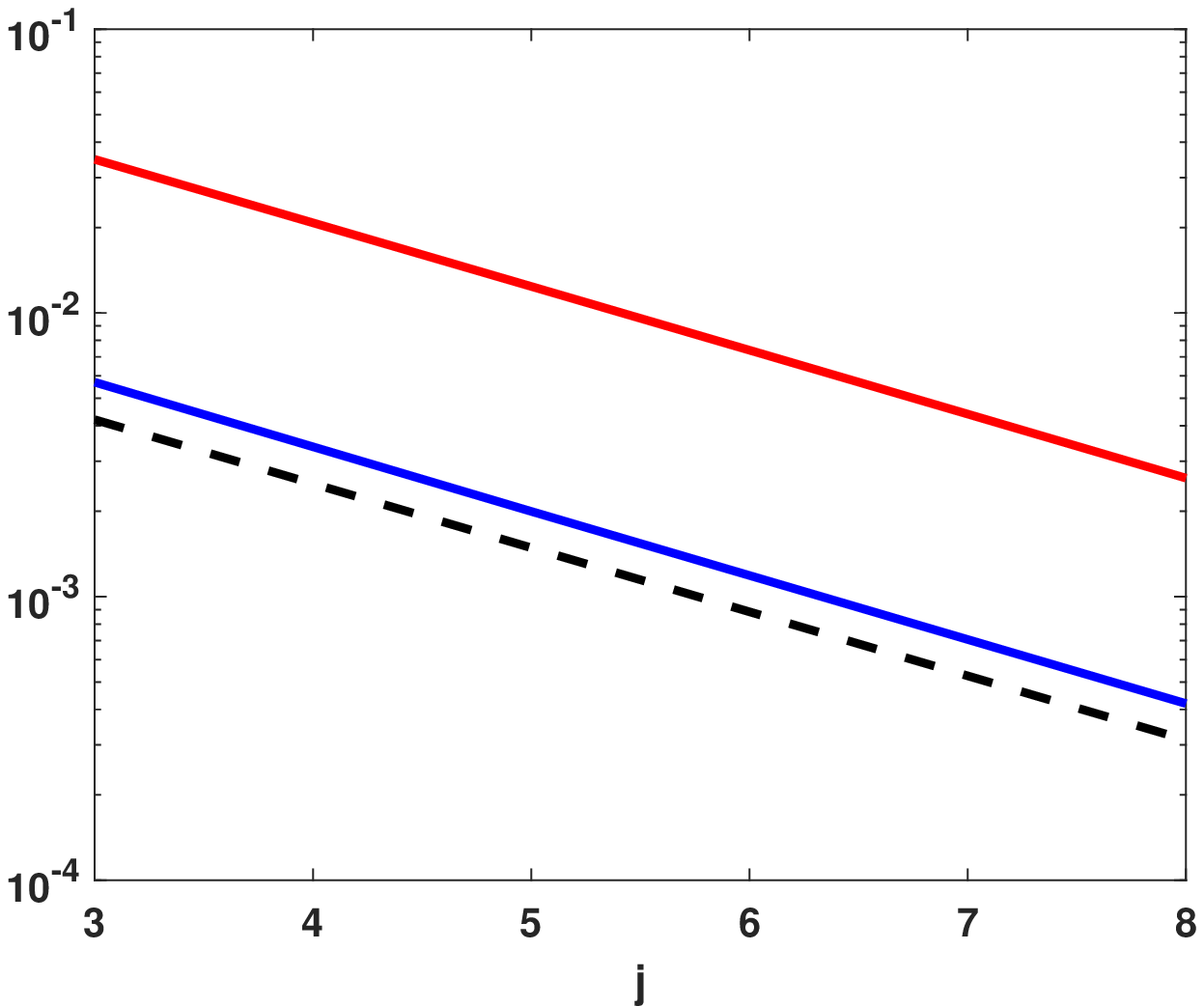} \\
	\end{tabular}
	\caption{The numerical convergence error $\rho_{\gamma,3}$ for $\gamma=0.10$ (left top panel), 0.25 (right top panel), 0.50 (left bottom panel), 0.75 (right bottom panel) and $n=3$ (red line), $n=4$ (blue line). The theoretical convergence order $2^{-j\gamma}$ is displayed as a black dashed line.}
	\label{Fig:ordine_gamma}
\end{figure}

\begin{figure} [ht]
	\centering
	\begin{tabular}{cc}
		\includegraphics[width=5cm]{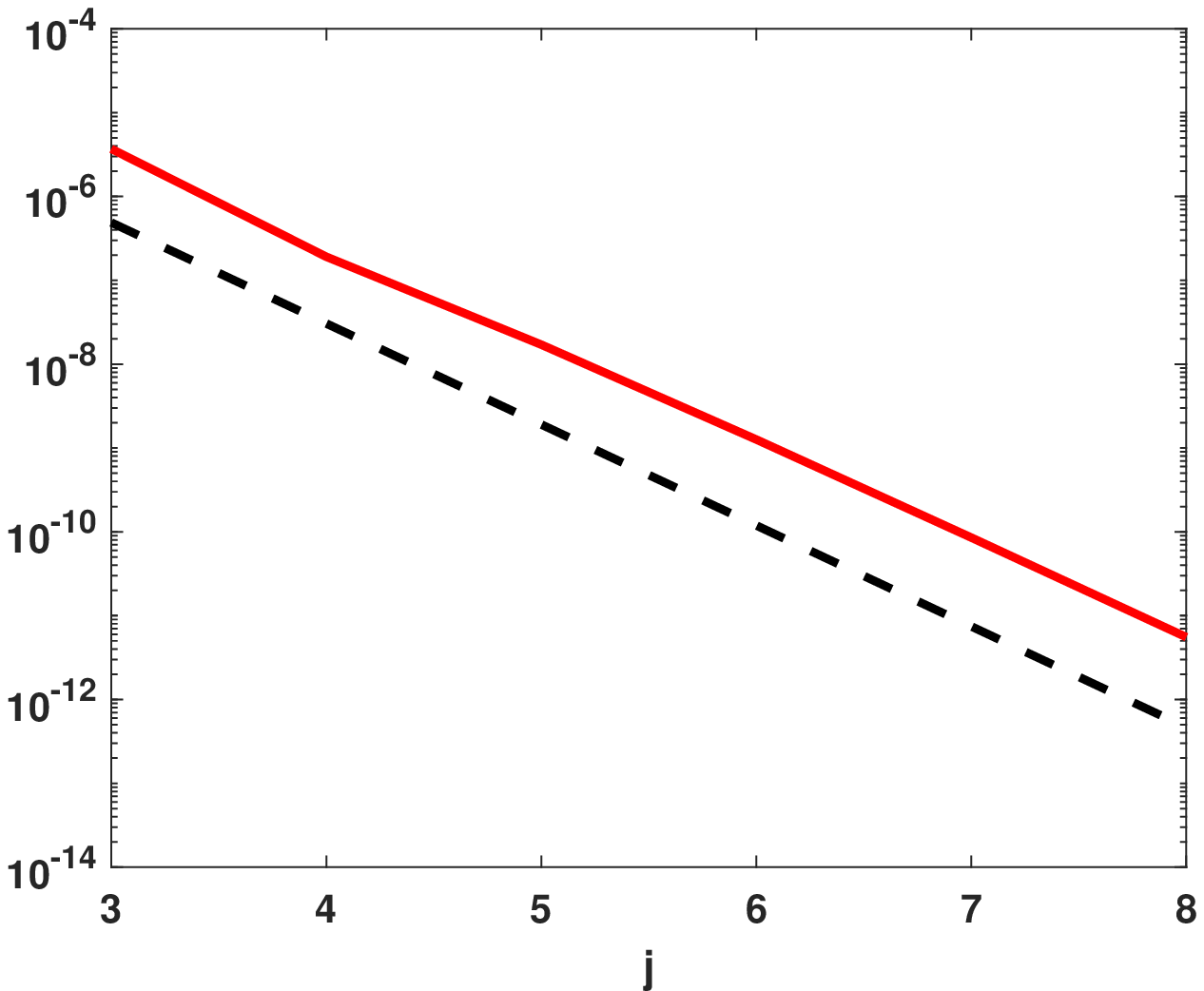} & \includegraphics[width=5cm]{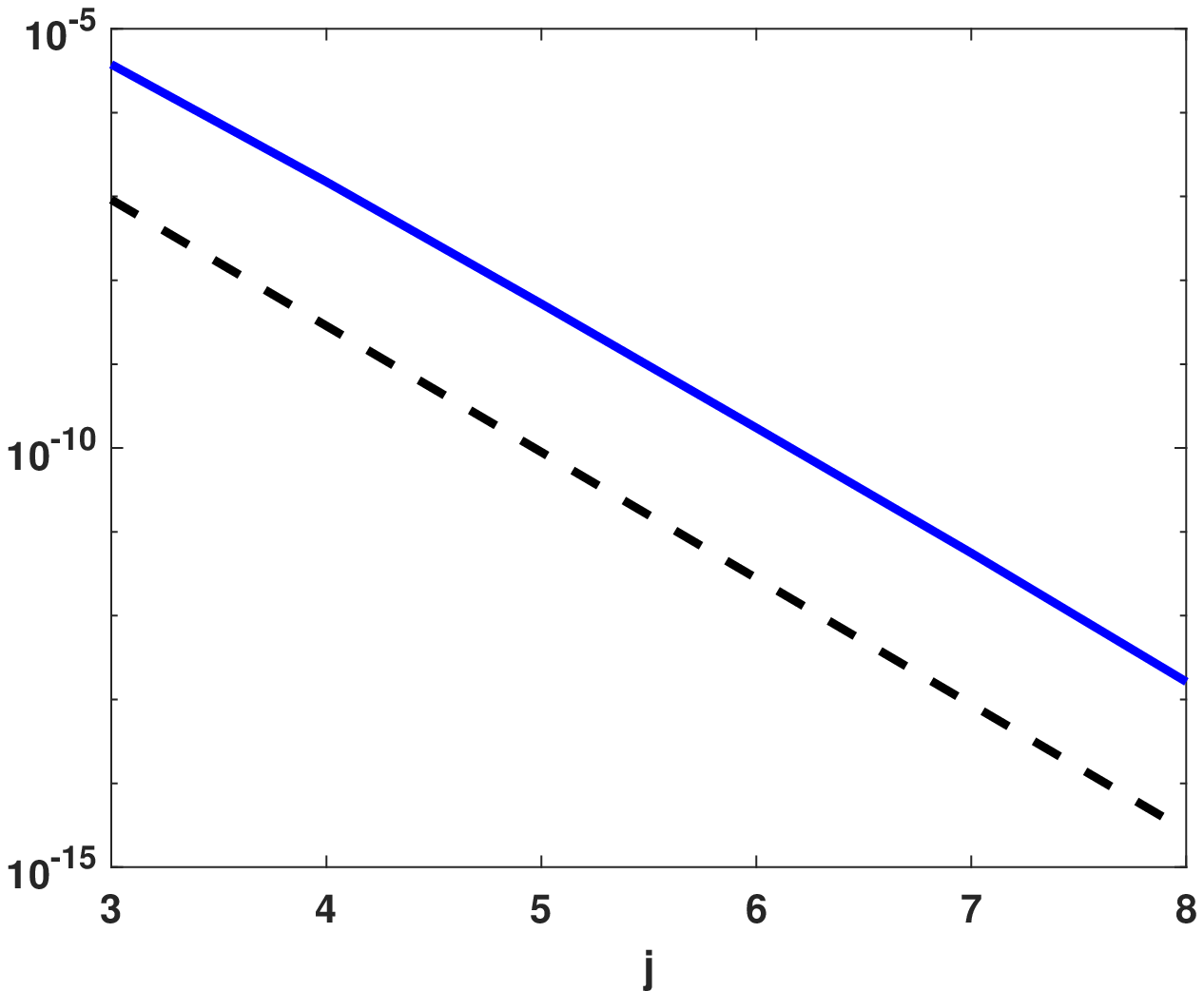} \\
	\end{tabular}
	\caption{The numerical convergence error $\rho_{1,n}$ for the classical problem with ordinary first derivative. Left panel refers to $n=3$ while right panel refers to $n=4$. The theoretical convergence order $2^{-j(n+1)}$ is displayed as a black dashed  line.}
	\label{Fig:ordine_1}
\end{figure}

\section{Conclusions}
\label{Sec:7}
We used a collocation method based on an interpolating projection operator on refinable polynomial spline spaces to approximate the solution of a linear fractional dynamical system.
We provide an explicit formula that allows us to evaluate the fractional derivatives of the approximating function in an accurate and easy way.  
The method can be used to solve several differential problems of fractional order and, in particular, nonlinear problems \cite{Pi18a,PiPe 2017} or boundary value problems \cite{PePi 2018.b}.
We notice that higher approximation order methods can be obtained by using different types of collocation points as in \cite{As78,KPT16,PeTa 11}. This will be the subject of a forthcoming paper.

\section*{Acknowledgment}
This work was partially supported by: grant of University of Roma ``La Sapienza'', Ricerca Scientifica 2017; INdAM-GNCS 2018 Project ``Sviluppo di modelli e metodi computazionali per l'elaborazione di segnali e immagini''.

\end{document}